\documentclass[11 pt]{amsart}
\usepackage{latexsym,amscd,amssymb, graphicx, amsthm}  
\usepackage{blkarray}
\usepackage{tikz-cd, xcolor}
\usetikzlibrary{calc}
\usepackage[margin=1in]{geometry}

\numberwithin{equation}{section}

\newtheorem{theorem}{Theorem}[section]
\newtheorem{proposition}[theorem]{Proposition}
\newtheorem{corollary}[theorem]{Corollary}
\newtheorem{lemma}[theorem]{Lemma}

\newtheorem{problem}[theorem]{Problem}
\newtheorem{example}[theorem]{Example}

\definecolor{2purple}{RGB}{204,102,255}
\definecolor{3green}{RGB}{0,204,0}

\newtheorem{defn}[theorem]{Definition}
\theoremstyle{definition}

\newcommand{\equicc}[4][1 cm]{
\draw (#3,#4) circle (#1);
\pgfmathparse{#1/1 cm+0.25};
\edef\oc{\pgfmathresult cm};
  \foreach \i in {1,...,#2} {
    \coordinate (N\i) at ($(#3,#4)+ (\i*360/#2:#1)$);
    \fill[black] (N\i) circle (0.05 cm);
    \draw ($(#3,#4)+ (\i*360/#2:\oc)$) node{$\i$};
  }
  \foreach \i in {2,...,#2} {
    \pgfmathparse{\i-1}
    \edef\j{\pgfmathresult}
  }
}

\newcommand{\inv}{{\mathrm {inv}}}

\newcommand{\symm}{{\mathfrak{S}}}

\begin{document}

\title[A combinatorial basis for the fermionic diagonal coinvariant ring]
{A combinatorial basis for the fermionic diagonal coinvariant ring}

\author{Jesse Kim}
\address
{Department of Mathematics \newline \indent
University of California, San Diego \newline \indent
La Jolla, CA, 92093-0112, USA}
\email{jvkim@ucsd.edu}

\begin{abstract}
Let $\Theta_n = (\theta_1, \dots, \theta_n)$ and $\Xi_n = (\xi_1, \dots, \xi_n)$ be two lists of $n$
 variables and consider the diagonal action of $\symm_n$ on the exterior algebra 
$\wedge \{ \Theta_n, \Xi_n \}$ generated by these variables. Jongwon Kim and Rhoades defined
and studied
the {\em fermionic diagonal coinvariant ring}  $FDR_n$ obtained from 
$\wedge \{ \Theta_n, \Xi_n \}$ by modding out by the $\symm_n$-invariants with vanishing constant term.
In joint work with Rhoades we gave a basis for the maximal degree components of this ring where the action of $\symm_n$ could be interpreted combinatorially via noncrossing set partitions. This paper will do similarly for the entire ring, although the combinatorial interpretation will be limited to the action of $\symm_{n-1} \subset \symm_n$. The basis will be indexed by a certain class of noncrossing partitions. 
\end{abstract}

\maketitle

\section{Introduction}
\label{Introduction}
This paper involves an algebraically defined $\mathfrak{S}_n$-module, and is concerned with modelling the $\mathfrak{S}_n$ action on this module via combinatorially defined objects. In particular, we will give a basis indexed by a certain type of noncrossing set partition for which the $\mathfrak{S}_n$ action has a nice combinatorial interpretation.

The module in question was introduced by Jongwon Kim and Rhoades \cite{KR}, and is defined as follows. Let $\Theta_n = (\theta_1, \dots, \theta_n)$ and $\Xi_n = (\xi_1, \dots, \xi_n)$ be two sets of $n$ anticommuting variables, and let
\begin{equation}
\wedge \{\Theta_n, \Xi_n \} := \wedge \{\theta_1, \dots, \theta_n , \xi_1, \dots, \xi_n \} 
\end{equation}
be the exterior algebra generated by these symbols over $\mathbb{C}$. The symmetric group $\mathfrak{S}_n$ acts on this exterior algebra via a diagonal action given by
\begin{equation}
w \cdot \theta_i := \theta_{w(i)} \qquad w \cdot \xi'_i := \xi'_{w(i)}.
\end{equation}
for any permutation $w\in \mathfrak{S}_n$ and $1 \leq i \leq n$. Let $\wedge \{\Theta_n, \Xi_n \}^{\mathfrak{S}_n}_+$ denote the subspace of $\mathfrak{S}_n$-invariants with vanishing constant term. Then the fermionic diagonal coinvariant ring is defined as 
\begin{equation}
FDR_n := \wedge \{\Theta_n, \Xi_n \}/\langle\wedge \{\Theta_n, \Xi_n \}^{\mathfrak{S}_n}_+ \rangle.
\end{equation}

The ring $FDR_n$ is a variant of the Garsia-Haiman diagonal coinvariant ring \cite{Haiman}, which is defined analogously but with the anticommuting variables replaced with commuting ones. Several other variants involving more sets of variables or mixtures of anticommuting and commuting variables have been studied by other authors \cite{Bergeron, BRT, DIW, KR, OZ, PRR, RW, RW2, Swanson, SW, ZabrockiDelta, ZabrockiFermion}.

The ring $\wedge \{\Theta_n, \Xi_n \}$ has a bigrading given by
\begin{equation}
(\wedge \{\Theta_n, \Xi_n \})_{i,j} := \wedge^i \{\theta_1, \dots, \theta_n \} \otimes \wedge^j \{\xi_1, \dots, \xi_n \} .
\end{equation}
The invariant ideal $\langle\wedge \{\Theta_n, \Xi_n \}^{\mathfrak{S}_n}_+ \rangle$ is homogeneous, so $FDR_n$ inherits the bigrading. In \cite{KR}, Kim and Rhoades calculated the frobenius image of $FDR_n$ to be given by
\begin{equation}
\textrm{Frob}(FDR_n)_{i,j}  = s_{(n-i, 1^i)} * s_{(n-j, 1^j)} - s_{(n-i-1, 1^{i+1})} * s_{(n-j-1, 1^{j+1})}
\end{equation}
where $*$ denotes the Kronecker product of Schur functions. They remark that in the case when $i+j = n-1$, the above shows that the dimension of $(FDR_n)_{n-k,k-1}$ is given by the Narayana number $\textrm{Nar}(n,k)$. Narayana numbers count noncrossing set partitions of $[n]$ into $k$ blocks, and in joint work with Rhoades \cite{me} we gave a combinatorial basis of $(FDR_n)_{n-k,k-1}$ indexed by set partitions for which the $\mathfrak{S}_n$-action was given by a skein action on noncrossing partitions first described by Rhoades in \cite{Rhoades}.

In this paper we will give a similar result for all bidegrees, although our results will not give a combinatorial description for the full $\mathfrak{S}_n$-action. Instead, we will focus on the subgroup  of $\mathfrak{S}_n$ consisting of permutations which leave $n$ fixed (which we will abusively refer to as $\mathfrak{S}_{n-1}$). We will define a basis of $(FDR_n)_{i,j}$ indexed by a certain class of noncrossing set partitions defined in Section 3 for which the action of $\mathfrak{S}_{n-1}$ can be described via combinatorial manipulations of the indexing partitions and use this basis to give an expression for the Frobenius image
\begin{equation}
\textrm{Frob}(\textrm{Res}^{\mathfrak{S}_n}_{\mathfrak{S}_{n-1}}(FDR_n)_{i,j}).
\end{equation}

The rest of the paper is organized as follows. Section 2 will give relevant background information on set partitions, exterior algebras, and $\mathfrak{S}_n$ representation theory. Section 3 will describe an action of  $\symm_{n-1}$ on certain set partitions and map this action into $FDR_n$. Section 4 will show that a restriction of this map is an isomorphism and use it to obtain a combinatorial basis of $FDR_n$. Section 5 will use the basis developed to calculate the bigraded $\symm_n$-structure of $FDR_n$. Section 6 will connect this basis to a cyclic sieving result of Thiel and address some avenues of further inquiry.

\section{Background}
\label{Background}
\subsection{Combinatorics}
A \textit{noncrossing set partition} of $[n]$ is a set partition of $[n]$ in which for any $1 \leq a<b<c<d \leq n$ if $a$ and $c$ are in the same block, and $b$ and $d$ are in the same block, then $a,b,c,d$ are all in the same block. 

An \textit{integer partition} $\lambda \vdash n$ of length $k$ is a sequence of integers $(\lambda_1, \lambda_2, \dots, \lambda_k)$ where $\lambda_1 + \cdots + \lambda_k = n$ and $\lambda_1 \geq \lambda_2 \geq \cdots \geq \lambda_k \geq 1$. \textit{Dominance order}, denoted by $\mu \preceq \lambda$, is a partial order on set partitions defined by $\mu \preceq \lambda$ if and only if $\mu_1 + \cdots + \mu_i \leq \lambda_1 + \cdots + \lambda_i$ for all $i$, taking $\mu_i$ or $\lambda_i$ to be 0 whenever $i$ exceeds the length of $\mu$ or $\lambda$ respectively. The conjugate of an integer partition $\lambda$ denoted $\lambda'$. 

\subsection{Exterior Algebras}
As in the introduction, we will use $\wedge \{\Theta_n, \Xi_n\}$ to denote the exterior algebra generated by the $2n$ symbols $\theta_1, \dots, \theta_n, \xi_1, \dots, \xi_n$. There is an isomorphism of graded vector spaces (see e.g. \cite{KR})
\begin{equation}
FDR_n \cong \wedge \{\theta_1, \dots, \theta_{n-1}, \xi'_1 , \dots, \xi'_{n-1}\} / \langle \theta_1\xi'_1 + \cdots + \theta_{n-1}\xi'_{n-1} \rangle
\end{equation}
given by
\begin{align*}
&\theta_i   \rightarrow \theta_i & 1\leq i \leq n-1\\
&\theta_n \rightarrow -(\theta_1 + \cdots + \theta_{n-1})&\\
&\xi_i \rightarrow \xi'_i - \frac{1}{n}(\xi'_1 + \cdots + \xi'_{n-1}) & 1 \leq i \leq n-1 \\
&\xi_n \rightarrow - \frac{1}{n}(\xi'_1 + \cdots + \xi'_{n-1}) 
\end{align*}
As this paper will focus on the action of $\mathfrak{S}_{n-1}$, we will extensively use this alternate formulation of $FDR_n$, and use $\wedge \{\Theta_{n-1}, \Xi'_{n-1}\}$ to denote $\wedge \{\theta_1, \dots, \theta_{n-1}, \xi'_1 , \dots, \xi'_{n-1}\}$. The ring $\wedge \{\Theta_{n-1}, \Xi'_{n-1}\}$ inherits the action of $\mathfrak{S}_n$ and $\mathfrak{S}_{n-1}$ from $FDR_n$, and the action of $\symm_{n-1}$ simply permutes indices of variables. 

Given subsets $S, T \subseteq [n-1]$, let $\theta_S, \xi'_T$ with $S = \{s_1 < \cdots < s_a\}$, $T = \{t_1 < \cdots < t_b\}$ denote the monomial
\begin{equation}
\theta_{s_1} \cdots \theta_{s_a}\cdot \xi'_{t_1} \cdots \xi'_{t_b}
\end{equation}
The set $\{\theta_S \cdot \xi'_T : S,T \subseteq [n-1]\}$ is a basis of $\wedge \{\Theta_{n-1}, \Xi'_{n-1}\}$. Define an inner product $\langle -,-\rangle$ on $\wedge \{\Theta_{n-1}, \Xi'_{n-1}\}$ by declaring this basis to be orthogonal.

An exterior algebra $\wedge \{ \omega_1, \dots, \omega_n\}$ acts on itself via {\em exterior differentiaiton}, denoted by $\odot$. The action $\odot: \wedge \{ \omega_1, \dots, \omega_n\} \times \wedge \{ \omega_1, \dots, \omega_n\} \rightarrow \wedge \{ \omega_1, \dots, \omega_n\}$ is defined by
\[
\omega_i \odot (\omega_{s_1}\cdots \omega_{s_k}) = \begin{cases} \omega_{s_1}\cdots\widehat{\omega_{s_j}}\cdots \omega_{s_k} & i = s_j \\ 0 & i \neq s_j \textrm{ for all } 1\leq j \leq k\end{cases}.
\]
\subsection{$\mathfrak{S}_n$-Representation Theory}
Irreducible representations of $\mathfrak{S}_n$ are in one-to-one correspondence with partitions $\lambda \vdash n$. Given $\lambda \vdash n$ let $S^\lambda$ denote the corresponding $\mathfrak{S}_n$-irreducible. Any finite dimensional $\mathfrak{S}_n$-module $V$ can be decomposed uniquely as $V \cong \bigoplus_{\lambda \vdash n} c_\lambda S_\lambda$ for some multiplicities $c_\lambda$. The \textit{Frobenius image} of $V$ is the symmetric function given by 
\begin{equation}
\textrm{Frob } V := \sum_{\lambda \vdash n} c_\lambda s_\lambda 
\end{equation}
where $s_\lambda$ is the Schur function corresponding to $\lambda$. 

If $V$ is an $\mathfrak{S}_n$-module and $W$ is an $\mathfrak{S}_m$-module, their \textit{induction product} $V\circ W$ is given by
\begin{equation}
V \circ W := \textrm{Ind}_{\mathfrak{S}_n \times \mathfrak{S}_m}^{\mathfrak{S}_{n+m}} (V \otimes W)
\end{equation}
where the action of $\mathfrak{S}_m \times \mathfrak{S}_n$ on $V\otimes W$ is given by $(\sigma, \sigma') \cdot (v \otimes w) := (\sigma \cdot v) \otimes (\sigma' \cdot w)$. We have 
\begin{equation}
\label{indprod}
\textrm{Frob } V\circ W = \textrm{Frob } V \cdot \textrm{Frob }W
\end{equation}
so induction product of modules corresponds to multiplication of Frobenius image.

Given a partition $\lambda = (\lambda_1, \lambda_2, \dots, \lambda_k) \vdash n$, let $\mathfrak{S}_\lambda \subseteq \mathfrak{S}_n$ denote the Young subgroup $\mathfrak{S}_\lambda := \mathfrak{S}_{\{1,\dots,\lambda_1\}} \times \mathfrak{S}_{\{\lambda_1+1,\dots,\lambda_1+\lambda_2\}}  \times \cdots \times \mathfrak{S}_{\{n-\lambda_k,\dots,n\}}$. To any subgroup $X \subseteq \mathfrak{S}_n$ we associate two group algebra elements $[X]_+$ and $[X]_-$ defined by $[X]_+ = \sum_{w\in X} w$ and $[X]_+ = \sum_{w\in X} \textrm{sign}(w)w$. We will need the following standard lemma.
\begin{lemma}
\label{symmetrizationkills}
Let $\lambda, \mu \vdash n$. Then $[S_\lambda]_+$ kills $S^\mu$ unless $\lambda \preceq \mu$ and $[S_{\lambda'}]_-$ kills $S^\mu$ unless $\mu \preceq \lambda$.
\end{lemma}

\subsection{Cyclic Sieving}
Cyclic sieving was introduced by Reiner, Stanton and White \cite{RSW} as a way to express various related results about the enumeration of fixed points of a cyclic action. If $X$ is a finite set, $C$ is a cyclic group of order $n$ generated by $c$ acting on $X$, and $P(q)$ is a polynomial in $\mathbb{N}[q]$, then we say that
\begin{defn}
The triple $(X,C, P(q))$ exhibits the cyclic sieving phenomenon if for all nonnegative integers $d$,
\[
|\{x\in X \mid c^d \cdot x = x \}| = P(\zeta^d)
\] 
where $\zeta$ is a primitive $n^{th}$ root of unity.
\end{defn}

The polynomial $P(q)$ is often given in terms of {\em $q$-analogs}. The $q$-analog of a positive integer $n$ is denoted $[n]_q$ and is defined to be $1+q+q^2+\cdots + q^{n-1}$. The $q$-analogs of $n!$, $\binom{n}{k}$ and the multinomial coefficient $\binom{n}{k_1,k_2, \dots, k_l}$ , are denoted and defined as follows:
\[
[n]_q! = [n]_q[n-1]_q\cdots[1]_q,
\]
\[
\begin{bmatrix} n\\k \end{bmatrix}_q = \frac{[n]_q!}{[n-k]_q![k]_q!},
\]
\[
\begin{bmatrix} n\\k_1, k_2, \dots k_l \end{bmatrix}_q = \frac{[n]_q!}{[k_1]_q!\cdots[k_l]_q!}.
\]

The {\em fake degree polynomial} of a representation is defined by 
\[
\textbf{fd}(S^{\lambda}) = q^{b(\lambda)}\frac{[r]_q!}{\prod_{(i,j \in \lambda)} [h(i,j)]_q}
\]
where the product is over all cells $(i,j)$ of $\lambda$, $h(i,j)$ is its hook length and $b(\lambda) = \lambda_2 + 2\lambda_3 + 3\lambda_4 + \cdots$ and the fake degree of a general representation is the sum of fake degrees of the irreducibles it contains, with multiplicity.

Cyclic sieving results are often proven via representation theory. In particular Reiner, Stanton and White \cite{RSW} realized the following was implied by a result of Springer's \cite{Springer}.
\begin{theorem} [Springer, 1974]
\label{fakedegree}
Let $V$ be a representation of $\mathfrak{S}_n$ with a basis $X$ which is preserved by the long cycle, $c$. Let $P(q) = \textbf{fd}(V)$. Then $(X, \langle c \rangle, P(q))$ exhibits the cyclic sieving phenomenon.
\end{theorem}

\section{Set partitions and the action of $\symm_{n-1}$}
\label{basis-section}
The indexing set for our combinatorial basis will be a certain partially labelled subset $\Phi(n)$ of noncrossing set partitions of $[n]$.
\begin{defn}
Let $n,k,x,t$ be nonnegative integers. We define the following sets of set partitions:
\begin{itemize}
\item Let $\Psi(n)$ denote the set of all set partitions of $n$ for which all blocks not containing $n$ are size 1 or size 2, and blocks of size 1 not containing $n$ are labelled with either a $\theta$ or a $\xi'$. \\
\item
Let $\Psi(n,k)$ be the set of partitions in $\Psi(n)$ in which the block containing $n$ is size $k$. \\
\item Let $\Psi(n,k,t,x)$ denote the set of partitions in $\Psi(n,k)$ which have exactly $t$ singletons labelled $\theta$ and exactly $x$ singletons labelled $\xi'$. \\
\item Let $\Phi(n)$, $\Phi(n,k)$, and $\Phi(n,k,t,x)$ be the subsets of $\Psi(n)$, $\Psi(n,k)$, or $\Psi(n,k,t,x)$ respectively which consist only of the those set partitions which are noncrossing.
\end{itemize}
\end{defn} 
For the rest of this paper, when we refer to the singleton blocks of a partition $\pi \in \Psi(n)$, we only refer to those blocks of size 1 that do not contain $n$, even if the block containing $n$ happens to be size 1. Similarly when we refer to the blocks of size two we refer to only the blocks of size two that do not contain $n$. 

There is a natural action of $\symm_{n-1}$ on $\Psi(n)$, given by simply permuting elements between blocks and preserving labels of blocks. The sets $\Psi(n,k)$ and $\Psi(n,k,x,t)$ are closed under this action, but $\Phi(n)$ is not, as permuting the elements of a noncrossing permutation may introduce crossings. However, we can define an action of $\symm_{n-1}$ on the linearization $\mathbb{C}\Phi(n)$ by mapping $\mathbb{C}\Psi(n)$ into $\wedge \{\Theta_{n-1}, \Xi'_{n-1}\}$ in such a way that $\mathbb{C}\Phi(n)$ is $\symm_{n-1}$-invariant and pulling back the $\symm_{n-1}$-action.

Towards this goal, to each element  $\pi \in \Psi(n)$ we will associate an element $G_\pi$ of $\wedge \{\Theta_{n-1}, \Xi'_{n-1}\}$. To define $G_\pi$ we will make use of a tool we will call  \emph{block operators}. Let $B$ be a block of a set partition $\pi \in \Psi(n)$, i.e. $B$ is a nonempty subset of $[n]$ that either contains $n$ or is size at most two. Define the \emph{block operator} $\tau_B : \wedge \{\Theta_{n-1}, \Xi'_{n-1}\} \rightarrow \wedge \{\Theta_{n-1}, \Xi'_{n-1}\}$ by 
\begin{equation}
\tau_B(f) = \begin{cases} (\prod_{i \in B\setminus \{n\}} \theta_i ) \odot f  & n \in B \\ \xi'_i \cdot(\theta_j \odot f) + \xi'_j \cdot (\theta_i \odot f) & n \not\in B, B = \{i,j\}\\ f & B = \{i_\theta\} \\ \xi'_i \cdot(\theta_i \odot f) & B = \{i_\xi'\} \end{cases}
\end{equation}
It will be important for what follows to note that block operators corresponding to blocks not containing $n$ commute
\begin{lemma}
\label{block-operator-commute}
Let $A$ and $B$ be two nonempty subsets of $[n-1]$ of size at most two. Then $\tau_A$ and $\tau_B$ commute. 
\end{lemma}
\begin{proof}
The lemma reduces to the fact that the family of operators $\{\xi'_1 \cdot, \dots, \xi'_{n-1} \cdot, \theta_1 \odot, \dots, \theta_{n-1} \odot \}$ all anticommute, and that each block operator is a degree two polynomial in these.
\end{proof}
Block operators also interact nicely with the action of $\mathfrak{S}_{n-1}$. 
\begin{lemma}
\label{blockoperatorequiv}
Let $A$ be a subset of $[n-1]$ and let $\sigma \in \mathfrak{S}_{n-1}$. Then for any $f\in \wedge\{\Theta,_{n-1}, \Xi'_{n-1}\}$
\[
\sigma \cdot \tau_A (f) = \tau_{\sigma \cdot A}( \sigma \circ f)
\]
where the action of $\mathfrak{S}_n$ on subsets is given by $\sigma \cdot \{a_1, \dots, a_k\} = \{\sigma(a_1), \dots, \sigma(a_k)\}.$ 
\end{lemma}
We can now define $G_\pi$.
\begin{defn} 
Let $\pi \in \Psi(n)$ with blocks $B_1, \dots ,B_k$ and $n \in B_k$. Then 
\begin{equation}
G_\pi := \tau_{B_1} \cdots \tau_{B_k} (\theta_1\theta_2 \cdots \theta_{n-1}).
\end{equation}
\end{defn}
We can also give a description of the $G_\pi$ not involving block operators as follows.
\begin{proposition}
Let $\pi \in \Psi(n)$. Take the product of $\theta_i\xi'_i - \theta_j\xi'_j$ for every size two block $\{i,j\}$ of $\pi$ with $i<j$. For each singleton block $\{i\}$ of $\pi$, multiply by $\theta_i$ or $\xi'_i$ according to its label in increasing order. Then $G_\pi$ is equal to the result multiplied by $(-1)^\textrm{inv}(\pi')$ where $\pi'$ is the word formed by listing all size two blocks not containing $n$ increasing within each block and by order of increasing minimal element, then listing all size one blocks not containing n in increasing order.
\end{proposition}
For example, if $\pi = 1_\theta / 2,5 / 3,4 / 5,6,8 / 7_\xi'$, then 
\begin{equation}
G_\pi = (-1)^{\inv(253417)}(\theta_2\xi'_2 - \theta_5\xi'_5)(\theta_3\xi'_3-\theta_4\xi'_4)\theta_1\xi'_7
\end{equation}
\begin{proof}
By Lemma~\ref{block-operator-commute} we can assume that all of the block operators corresponding to size two blocks appear before block operators according to singletons. Applying $\tau_{B_k}$ and any block operators corresponding to singletons to $ (\theta_1\theta_2 \cdots \theta_{n-1})$ removes all $\theta_i$ indexed by elements of $B_k$ and replaces $\theta_i$ indexed by $\xi'$-labelled singletons with $\xi'_i$. Note that $\tau_{\{i,j\}} \theta_i \theta_j = \theta_i\xi'_i - \theta_j\xi'_j$, and the proof follows.
\end{proof}

The $\symm_{n-1}$ action on these $G_\pi$ matches the natural $\symm_{n-1}$ action on $\Psi(n)$, up to sign.
\begin{proposition}
Let $\sigma \in \mathfrak{S}_{n-1}$ and $\pi \in \Psi(n)$. Then $\sigma \circ G_\pi = \mathrm{sign}(\sigma) G_{\sigma \circ \pi}$.
\end{proposition}
\begin{proof}
Using the block operator definition of $G_\pi$ and Lemma~\ref{blockoperatorequiv} we have,
\begin{align}
\sigma \circ G_\pi &= \sigma \circ (\tau_{B_1} \cdots \tau_{B_k} (\theta_1\theta_2 \cdots \theta_{n-1}))\\ & = \tau_{\sigma(B_1)} \cdots \tau_{\sigma(B_k)}( \sigma \circ (\theta_1\theta_2 \cdots \theta_{n-1})) \\ &= \tau_{\sigma(B_1)} \cdots \tau_{\sigma(B_k)}( \textrm{sign} (\sigma) \theta_1\theta_2 \cdots \theta_{n-1}) \\ &= \textrm{sign}(\sigma) G_{\sigma \circ \pi}
\end{align}
\end{proof}

The goal of the remainder of this section is to show that $\textrm{span} (\{ G_\pi \mid \pi \in \Phi(n)\})$ is $\symm_{n-1}$ invariant. For this end we will need the following two relations of block operators.
\begin{lemma}
\label{skein}
Let $a,b,c,d \in [n-1]$. Then 
\begin{equation}
\tau_{\{a,b\}}\tau_{\{c,d\}} +  \tau_{\{a,c\}}\tau_{\{b,d\}} +\tau_{\{a,d\}}\tau_{\{b,c\}} = 0
\end{equation}
as operators on the ring $\wedge \{\Theta_{n-1}, \Xi'_{n-1}\}$
\end{lemma}
\begin{proof}
This is a straightforward calculation from the definition of $\tau$.
\end{proof}
\begin{lemma}
\label{2-cross-more-than-2}
Let $A = \{a_1 < a_2\} \subset [n-1]$ and $B \subset [n]$ be two disjoint sets with $n \in B$. Let $b_1<b_2<\cdots < b_m$ be the elements of $B$ that lie between $a_1$ and $a_2$, and suppose at least one such element exists. Then
\begin{equation}
\tau_A \tau_B + \tau_{\{a_1,b_1\}} \tau_{B + a_2 - b_1}  + \sum_{i=1}^{m-1} \tau_{\{b_i, b_{i+1}\}} \tau_{B+a_1+a_2- b_i-b_{i+1}} +\tau_{\{b_m, a_2\}} \tau_{B+a_1-b_m} = 0
\end{equation}
as operators on the ring $\wedge \{\Theta_{n-1}, \Xi'_{n-1}\}$
\end{lemma}
\begin{proof}
For any two cyclically consecutive elements $c_1, c_2$ of $a_1, b_2, \dots, b_m ,a_2$, and any third element $c_3 \in B \cup \{a_1, a_2\}$, the terms $\xi'_{c_1} \cdot \theta_{c_2} \odot \theta_{c_3} \odot$ and $\xi'_{c_1} \cdot \theta_{c_3} \odot \theta_{c_2} \odot$ will appear in the expansion of left hand side both exactly once or both exactly twice, depending on whether $c_3$ is also cyclically consecutive with $c_1$. In either case, anticommutativity results in the sum being 0.
\end{proof}
Together these lemmas allow us to demonstrate the $\symm_{n-1}$ invariance via a combinatorial algorithm.
\begin{corollary}
\label{algorithm}
Let $\sigma \in \mathfrak{S}_{n-1}$ and let $\pi \in \Phi(n)$. Then $\sigma \cdot G_\pi$ can be expressed as a linear combination of $\{G_\pi \mid \pi \in \Phi(n)\}$ via the following algorithm:
\begin{enumerate}
\item Apply $\sigma$ to $\pi$, resulting in a set partition $\pi'$ not necessarily in $\Phi(n)$.\\
\item If $\pi'$ is contains any crossing two element blocks $\{a,c\}$, $\{b,d\}$, neither of which contain $n$, replace $\pi'$ with minus the sum of the partitions obtained by replacing $\{a,c\}, \{b,d\}$ with $\{a,b\},\{c,d\}$ and $\{a,d\}, \{b,c\}$. Repeat on each new term of the sum until all terms of the sum do not contain crossing two element blocks.\\
\item For each term of the sum obtained in step 2, replace any two element set that crosses the block containing $n$ as described by Lemma~\ref{2-cross-more-than-2}.\\
\item Replace each partition $\pi''$ in the sum obtained from step 3 with its corresponding $G_{\pi''}$ to express $\sigma \cdot G_\pi$ as a linear combination.
\end{enumerate}
\end{corollary}

\begin{example}
Let $n=8$ and let $\sigma \in \symm_{n-1}$ be the cycle $(3576)$. Let $\pi \in \Phi(n)$ be the set partition $\{23/45/7_\theta/186\}$. An example of applying Corollary~\ref{algorithm} to this situation is given in Figure 1.

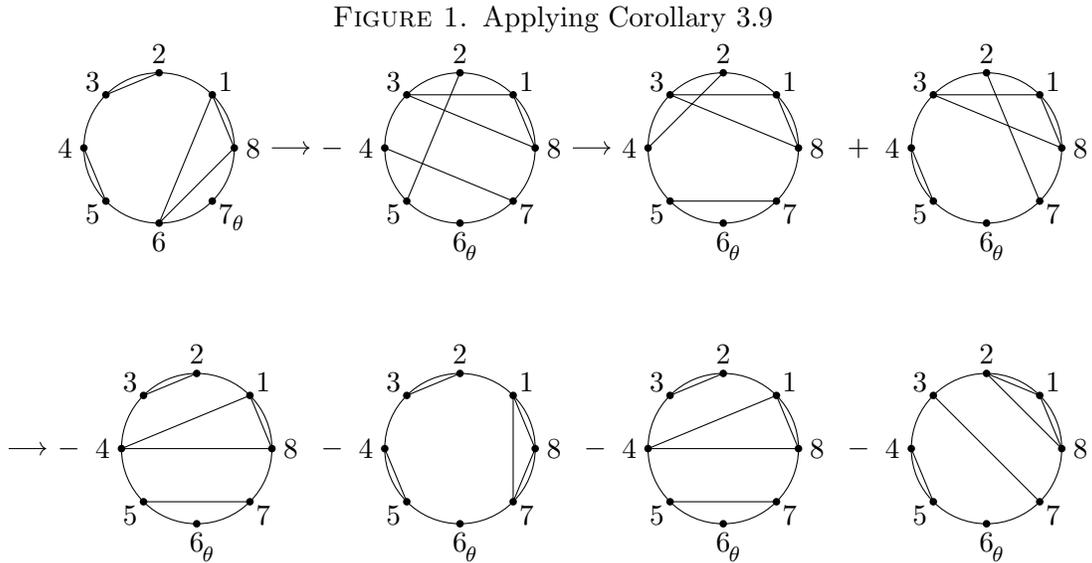
\begin{center}
\begin{figure}[h]
\caption{Applying Corollary~\ref{algorithm}}
\begin{tikzpicture}
\equicc[1 cm]{8}{0}{0}
\draw (N2)--(N3);
\draw (N4)--(N5);
\draw (N1)--(N8)--(N6)--(N1);
\draw[->] (1.5,0)--(2,0);
\node at (1.05,-1) {\scriptsize $\theta$};

\equicc[1 cm]{8}{4}{0}
\draw (N2)--(N5);
\draw (N4)--(N7);
\draw (N1)--(N3)--(N8)--(N1);
\node at (2.3,0) {$-$};
\node at (4.15,-1.4) {\scriptsize $\theta$};

\draw[->] (5.5,0)--(6,0);

\equicc[1 cm]{8}{7.5}{0}
\draw (N2)--(N4);
\draw (N1)--(N3)--(N8)--(N1);
\draw (N5)--(N7);
\node at (7.65,-1.4) {\scriptsize $\theta$};

\equicc[1 cm]{8}{11}{0}
\draw (N4)--(N5);
\draw (N2)--(N7);
\draw (N1)--(N3)--(N8)--(N1);
\node at (9.3,0) {$+$};
\node at (11.15,-1.4) {\scriptsize $\theta$};

\draw[->] (-2,-4)--(-1.5,-4);

\equicc[1 cm]{8}{0.5}{-4}
\draw (N7)--(N5);
\draw (N2)--(N3);
\draw (N1)--(N4)--(N8)--(N1);
\node at (-1.2,-4) {$-$};
\node at (.65,-5.4) {\scriptsize $\theta$};
\equicc[1 cm]{8}{4}{-4}
\draw (N2)--(N3);
\draw (N4)--(N5);
\draw (N1)--(N7)--(N8)--(N1);
\node at (2.3,-4) {$-$};
\node at (4.15,-5.4) {\scriptsize $\theta$};

\equicc[1 cm]{8}{7.5}{-4}
\draw (N7)--(N5);
\draw (N2)--(N3);
\draw (N1)--(N4)--(N8)--(N1);
\node at (5.8,-4) {$-$};
\node at (7.65,-5.4) {\scriptsize $\theta$};

\equicc[1 cm]{8}{11}{-4}
\draw (N4)--(N5);
\draw (N7)--(N3);
\draw (N1)--(N2)--(N8)--(N1);
\node at (9.3,-4) {$-$};
\node at (11.15,-5.4) {\scriptsize $\theta$};
\end{tikzpicture}
\end{figure}
\end{center}

\end{example}

\section{A combinatorial basis}
We have shown that there is a mapping of $\symm_{n-1}$-modules $\mathbb{C}\Psi(n) \rightarrow \wedge \{\Theta_{n-1}, \Xi'_{n-1}\}$. In this section we will show that the restriction of this mapping to $\mathbb{C}\Phi(n)$ is injective and becomes an isomorphism when composed with the quotient map $\wedge \{\Theta_{n-1}, \Xi'_{n-1}\} \rightarrow FDR_n$, thereby proving the following.

\begin{theorem}
\label{basis}
The set $\{ [G_\pi] \mid \pi \in \Psi(n) \}$ forms a basis for $FDR_n$, where $[f]$ denotes the equivalence class in $FDR_n$ of $f \in \wedge  \{\Theta_{n-1}, \Xi'_{n-1}\}$.
\end{theorem}

\begin{proof}
We begin with a dimension count; Kim and Rhoades \cite{KR} gave a basis of $FDR_n$ indexed by a set $\Pi(n)_{>0}$ of Motzkin-like lattice paths defined as follows. 
\begin{defn}
Let $\Pi(n)_{>0}$ be the set of all lattice paths which 
start at $(0,0)$, take steps $(1,0), (1,1)$ or $(1,-1)$, only touch the $x$-axis at $(0,0)$ and have all $(1,0)$ steps labelled by $\theta$ or $\xi'$.
\end{defn}
The two indexing sets are in bijection.
\begin{lemma}
\label{path-partition-bijection}
There is a bijection between $\Pi(n)_{>0}$ and $\Phi(n)$. 
\end{lemma}
\begin{proof}
Given a Motzkin path in $\Pi(n)_{>0}$, draw a horizontal line extending to the right of each up step until it first intersects the path again. Label each step after the first $1$ to $n-1$. Construct a set partition by placing every up step in a block with the down step it is connected to if such a down step exists, or in the block containing $n$ otherwise. Place every horizontal step in a singleton block with the same label. The process can be reversed, and is therefore a bijection. 
\end{proof}
The bijection is best described with a picture example as in Figure 2. 
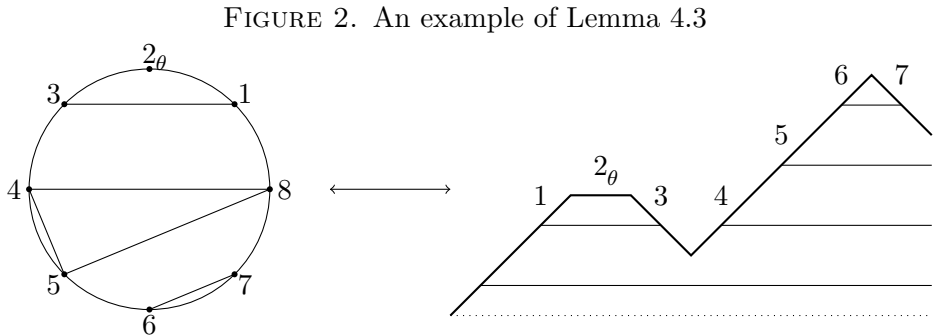
\begin{center}
\begin{figure}[h]
\caption{An example of Lemma~\ref{path-partition-bijection}}
\begin{tikzpicture}[scale = .8]
    \equicc[2 cm]{8}{-5}{2.1}
\draw (N1) -- (N3) ;
\draw (N8) -- (N5) -- (N4) -- (N8);
\draw (N6) -- (N7);
\node at (-4.8,4.2) {\scriptsize $\theta$};

\draw [to-to] (-2,2.1) -- (0,2.1);

\draw [thick] (0,0) -- (1,1) -- (2,2) -- (3,2) -- (4,1) -- (5,2) -- (6,3) -- (7,4) -- (8,3);
\draw [dotted] (0,0) -- (8,0);
\draw (.5, .5) -- (8,.5);
\draw (1.5,1.5) -- (3.5,1.5);
\draw (4.5,1.5) -- (8,1.5);
\draw (5.5, 2.5) -- (8,2.5);
\draw (6.5, 3.5) -- (7.5, 3.5);
\node at (1.5,2) {1};
\node at (2.5,2.5) {2};
\node at (2.7,2.3) {\scriptsize $\theta$};
\node at (3.5,2) {3};
\node at (4.5,2) {4};
\node at (5.5,3) {5};
\node at (6.5,4) {6};
\node at (7.5,4) {7};
\end{tikzpicture}
\end{figure}
\end{center}

Therefore it suffices to show that $\{ [G_\pi] \mid \pi \in \Psi(n) \}$ spans. By Corollary~\ref{algorithm} and since $FDR_n$ is defined as a quotient, it suffices to show that together, the sets
\[
\beta := \{ G_\pi \mid \pi \in \Psi(n) \}
\] and
\[ \beta' := \{m(\theta_1\xi'_1 + \cdots + \theta_{n-1}\xi'_{n-1}) \mid m \textrm{ a monomial in } \wedge \{\Theta_{n-1}, \Xi'_{n-1}\}\}
\] span
\[
 \wedge \{\Theta_{n-1}, \Xi'_{n-1}\}. 
\]
To show that $\beta \cup \beta'$ spans, we will break $\wedge \{\Theta_{n-1}, \Xi'_{n-1}\}$ into many subspaces and show that each subspace is spanned.

Let $S$ be a subset of $[n-1]$ of size $2k$ for some integer $k$. Let $m$ denote a fixed monomial of $\wedge \{\Theta_{n-1}, \Xi'_{n-1}\}$ such that the following conditions hold for all $i \in [n-1]$:
\begin{enumerate}
\item If $i \in S$, then neither $\xi'_i$ nor $\theta_i$ appears in $m$.
\item If $\xi'_i$ appears in $m$, then $\theta_i$ does not appear in $m$.
\end{enumerate}
Let $V_{S,m}$ denote the subspace of $\wedge \{\Theta_{n-1}, \Xi'_{n-1}\}$ spanned by monomials of the form $\theta_{s_1}\xi'_{s_1} \cdots \theta_{s_k}\xi'_{s_k}m$, where $s_1, \dots, s_k \in S$. There are $\binom{2k}{k}$ such monomials, so 
\begin{equation}
\textrm{dim}(V_{S,m}) = \binom{2k}{k}.
\end{equation} Consider the set of $\beta \cap V_{S,m}$. These will consist of all $\pi \in \Psi(n)'$ such that the size 1 parts of $\pi$ and their labels correspond exactly with the monomial $m$, and the size two parts partition $S$. This set is therefore in bijection with noncrossing perfect matchings of $S$, so we have
\begin{equation}
| \beta \cap V_{S,m}| = \textrm{Cat}(k)
\end{equation}
 where $\textrm{Cat}(k)$ is the $k$th Catalan number. Consider as well the set $\beta' \cap V_{S,m}$. If $m'$ is a degree $n-3$ monomial such that $(\theta_1\xi'_1 + \cdots + \theta_{n-1}\xi'_{n-1})m' \in V_{S,m}$, then it must be the case that $m' = \theta_{s_1}\xi'_{s_1}\cdots\theta_{s_{k-1}}\xi'_{s_{k-1}}m$ for some choice of $s_1, \dots, s_{k-1} \in S$. So we have
\begin{equation}
| \beta' \cap V_{S,m}| = \binom{2k}{k-1}.
\end{equation}
Putting the above equations together we have
\begin{equation}
|(\beta \cup \beta') \cap V_{S,m} | = \textrm{Cat}(k) + \binom{2k}{k-1} = \binom{2k}{k} = \textrm{dim}(V_{S,m})
\end{equation}
and so it suffices to show that $(\beta \cup \beta') \cap V_{S,m}$ is a linearly independent set.

Let $d$ be the degree of $m$, and let $M$ be the set of monomials of degree $d+2k$ in $\wedge \{\Theta_{n-1}, \Xi'_{n-1}\}$ whose variables are in increasing numerical order with $\theta_1 < \xi_1 < \cdots < \theta_n < \xi_n$. Define an inner product $ \langle - , - \rangle $ on the degree $d+2k$ part of $\wedge \{\Theta_{n-1}, \Xi'_{n-1}\}$ such that $M$ is an orthonormal set. With respect to this inner product, 
\begin{equation}
 \beta \cap V_{S,m} \subseteq (\beta' \cap V_{S,m})^\perp
\end{equation} To see this, suppose that $f_\pi \in V_{S,m}$ and $(\theta_1\xi'_1 + \cdots + \theta_{n-1}\xi'_{n-1})m' \in V_{S,m}$ have monomials in common. Then $m'$ must be equal to $ \theta_{s_1}\xi'_{s_1}\cdots\theta_{s_{k-1}}\xi'_{s_{k-1}}m'$ where each $s_i$ is in a distinct size 2 part of $\pi$. If this is the case, then $ \theta_{s_1}\xi'_{s_1}\cdots\theta_{s_{k-1}}\xi'_{s_{k-1}}m'$ and $f_\pi$ share exactly two monomials, corresponding to the two elements in the last size 2 part of $\pi$. These monomials will have coefficients of opposite sign in $f_\pi$ and the same sign in $ \theta_{s_1}\xi'_{s_1}\cdots\theta_{s_{k-1}}\xi'_{s_{k-1}}m'$, so the inner product will be 0. Therefore it suffices to show that $\beta \cap V_{S,m}$ and $\beta' \cap V_{S,m}$ are both individually linearly independent sets. 

To see that $\beta \cap V_{S,m}$ is linearly independent, consider the lexicographic term order on monomials with respect to the variable order $\theta_1, \xi'_1, \theta_2, \xi'_2, \dots$. With respect to this order, the leading term of $f_\pi$ is $\theta_{s_1}\xi'_{s_1}\cdots\theta_{s_k}\xi'_{s_k}m$, where $s_1, \dots, s_k$ are the numerically smaller elements of each size two block of $\pi$. Since $\pi$ is noncrossing and $m$ and $S$ determine the singletons and block containing $n$, specifying the set of elements that are the smaller of their part uniquely determines $\pi$. Therefore the $f_\pi$ contained in $V_{S,m}$ all have unique leading terms and are therefore linearly independent.

Kim and Rhoades proved \cite{KR} that in $FDR_n$, multiplication by $\theta_1\xi'_1 + \cdots + \theta_{n-1}\xi'_{n-1}$ is an injection, so $\beta' \cap V_{S,m}$ is also a linearly independent set and $V_{S,m}$ is spanned by $(\beta \cup \beta') \cap V_{S,m}$. Since every monomial is contained in some $V_{S,m}$, we therefore have that $\wedge \{\Theta_{n-1}, \Xi'_{n-1}\}$ is spanned by $\beta \cup \beta'$ and therefore $\{ [G_\pi] \mid \pi \in \Psi(n) \}$ is a basis for $FDR_n$ as desired.
\end{proof}

\section{$\mathfrak{S}_{n-1}$ module structure}
\label{structure}
In this section we will describe the Frobenius image of each bigraded piece of $FDR_n$ as an $\mathfrak{S}_{n-1}$ module. Consider the family of subspaces:
\begin{equation}
V(n,k,x,y) := \textrm{span} \{ [G_\pi] \mid \pi \in \Phi(n, k , x ,y ) \} \subseteq FDR_n
\end{equation}
These subspaces are in fact submodules of $\textrm{Res}^{\mathfrak{S}_n}_{\mathfrak{S}_{n-1}}(FDR_n)$, since they are closed under the action of $\mathfrak{S}_{n-1}$. To see this, note that no step of the algorithm described in Corollary~\ref{algorithm} replaces a set partition with one with a different number of size two blocks, $\xi'$-labelled elements, or $\theta$-labelled elements. Since $\Phi(n) = \oplus_{k,x,y} \Phi(n,k,x,y)$ the subspaces $V(n,k,x,y)$ make up all of $FDR_n$:
\begin{proposition}
\label{fdrsum}
The $i,j$-graded piece of $DR_n$ is a direct sum of $V(n,k,x,y)$:
\[
(FDR_n)_{i,j} = \bigoplus_{\substack{k,x,y\\k+x = i \\ k+y = j}} V(n,k,x,y)
\]
\end{proposition}
\begin{proof}
From the definition of $G_\pi$ it is clear that if $\pi \in \Phi(n,k,x,y)$ then $G_\pi$ has bidgree $(k+x,k+y)$. The result follows.
\end{proof}

To determine the structure of these modules we begin with $V(n,k,0,0)$. We first need a lemma
\begin{lemma}
\label{bijection}
There exists a bijection from $\Phi(n,k,0,0)$ to $SYT(n-k-1,k)$, the set of standard Young tableau of shape $\lambda = (n-k-1, k)$. 
\end{lemma}
\begin{proof}
Define a function $g: \Phi(n,k,0,0) \rightarrow \binom{[n-1]}{[k]}$ by
\begin{equation}
g(\pi) = \{i \in [n-1] \mid i \textrm{ is in a block of size 2, and is the larger element in its block}.\}
\end{equation}
For example, $g(14/23/78/569) = \{3,4,8\}$. Then $g$ is injective, it is possible to recover the preimage of a set $S$ under $g$ by starting with the smallest $i$ element of $S$, if $g(\pi) = S$, then for $\pi$ to satisfy the noncrossing condition, $\{i-1, i\}$ must be a block of $\pi$. Then the next smallest element of $S$ must be paired with the largest element smaller than it that is not already paired, and so on. This algorithm will produce a unique preimage iff $S$ satisfies the condition that for any $k \in [n-1]$, $|S \cap [k]| \leq k/2$.
Define another function $h: SYT(n-k-1, k) \rightarrow \binom{[n-1]}{k}$ by
\begin{equation}
h(T) = \{ i \in [n-1] \mid i \textrm{ is in the second row of } T\}
\end{equation}
Then $h$ is also injective, and $S \in h(SYT(n-k-1,k))$ iff $S$ satisfies the condition that for any $k \in [n-1]$, $|S \cap [k]| \leq k/2$. So the image of $h$ and $g$ are the same and the result follows.
\end{proof}
\begin{proposition}
We have that $V(n,k,0,0) \cong_{\mathfrak{S}_{n-1}} S^{(n-k-1, k)}$. 
\end{proposition}
\begin{proof}
Let $\lambda = (n-k-1,k)$. By Theorem~\ref{basis} and Lemma~\ref{bijection}, the dimensions of the modules agree, so by Lemma~\ref{symmetrizationkills} it suffices to show that $[\mathfrak{S}_\lambda]_+$ does not kill $V(k,0,0)$, but $ [\mathfrak{S}_\mu]_+$ does kill $V(n,k,0,0)$ for all partitions $\mu \succ \lambda$.

We begin by showing that $[\mathfrak{S}_\lambda]_+$ does not kill $V(k,0,0)$. Let $\pi_0 \in \Phi(n,k,0,0)$ be the parition whose blocks are
\[
\{n-1,n-2k\}, \{n-2,n-2k+1\} \{n-3,n-2k+2\}, \dots, \{n-k, n-k-1\}, \{1,2,3, \dots, n-2k-1, n\}
\]
Then using the block operator definition of $F_{\pi_0}$ we have
\begin{equation}
\label{idk}
[\mathfrak{S}_\lambda]_+ F_{\pi_0} = \sum_{\sigma \in \mathfrak{S}_\lambda} \sigma \cdot \tau_{\{n-1, n-2k\}} \cdots \tau_{\{n-k, n-k-1\}} \tau_{\{1,2,3,\dots, n-2k-1, n\}} \theta_{1}\cdots \theta_{n-1}
\end{equation}
Consider the coefficient of $\theta_{n-1} \cdots \theta_{n-k}\xi'_{n-1} \cdots \xi'_{n-k}$ in the above expression. For a term to contribute to this coefficient, it must be the case that $\sigma \cdot \{1,2,3, \dots, n-2k-1, n\}  = \{1,2,3, \dots, n-2k-1, n\}$. If this is the case, then the summand corresponding to $\sigma$ can be written as
\begin{equation}
\label{sigmaprime}
 \tau_{\{n-1, \sigma'(n-2k)\}} \cdots \tau_{\{n-k, \sigma'(n-k-1)\}}  \theta_{n-2k}\cdots \theta_{n-1}
\end{equation}
for some permutation $\sigma'$ of $\{n-k-1, n-2k\}$. The coefficient of $\theta_{n-1} \cdots \theta_{n-k}\xi'_{n-1} \cdots \xi'_{n-k}$ in equation~\ref{sigmaprime} above does not depend on $\sigma'$, so all terms of the sum in equation~\ref{idk} which contribute to the coefficient of $\theta_{n-1} \cdots \theta_{n-k}\xi'_{n-1} \cdots \xi'_{n-k}$ contribute the same sign, and thus the coefficient of $\theta_{n-1} \cdots \theta_{n-k}\xi'_{n-1} \cdots \xi'_{n-k}$ in $[\mathfrak{S}_\lambda]_+ F_{\pi_0}$ is nonzero. Thus $V(k,0,0)$ is not killed by $[\mathfrak{S}_\lambda]_+$.

Now let $\mu$ be any partition of $n-1$ such that $\lambda \succ \mu$, i.e. $\mu = (n-m, m-1)$ for any $m\leq k$. Let $\pi \in \Phi(n,k,0,0)$. Since $m-1 < k$, there must be at least two elements of $i$ and $j$ of $[n-m]$ in the same block in $\pi$. Then the transposition $(i, j)$ acts on $G_\pi$ via multiplication by $-1$, so $(1 + (i, j)) G_\pi = 0$. But $ [\mathfrak{S}_\lambda]_+ = A(1+(i,j)$ for some symmteric group algebra element $A$, so indeed $[\mathfrak{S}_\lambda]_+ G_\pi = 0$, and the result follows.
\end{proof}
We can use $V(n,k,0,0)$ to determine the structure of $V(n,k,x,y)$ for any $x,y$. 
\begin{proposition}
\label{vfrob}
We have that 
\[
V(n,k,x,y) \cong_{\mathfrak{S}_{n-1}} \mathrm{Ind}_{\mathfrak{S}_{n-x-y-1} \otimes \mathfrak{S}_{x} \otimes \mathfrak{S}_{y}}^{\mathfrak{S}_{n-1}} S^{(n-x-y-k-1, k)} \otimes \mathrm{sign}_{\mathfrak{S}_x} \otimes \mathrm{sign}_{\mathfrak{S}_y}.
\]
\end{proposition}
\begin{proof}
We can represent an element $\pi$ of $\Phi(n,k,x,y)$ by the triple $(X,Y,\pi')$, where $X$ is the set of singletons labelled by $\xi'$, $Y$ is the set of singletons labelled by $\theta$, and $\pi'$ is the set partition obtained by removing all singletons from $\pi$ and decrementing indices. Let $F_{(X,Y,\pi')}$ denote $G_\pi$ for the corresponding $\pi$. The action of a transposition $(i,j)$ on $F_{(X,Y, \pi')}$ is then given by 
\begin{equation}
(i,j) \circ F_{(X,Y, \pi')} = \begin{cases} -F_{(X,Y,\pi')} & \{i,j\} \subset X \textrm{ or } \{i,j\} \subset Y \\ F_{(X,Y, (i,j) \circ \pi')} & \{i,j\} \subset (X \cup Y)^c \\ F_{(i,j) \circ X, (i,j) \circ Y, \pi'} \textrm{ otherwise }\end{cases}
\end{equation}
The proposition follows from the definition of induced representation.
\end{proof}
\begin{corollary}
\label{frobimg}
The Frobenius image of $V(n,k,x,y)$ is given by $s_{(n-x-y-k-1, k)}s_{(1^x)}s_{(1^y)}$. The Frobenius image of $(FDR_n)_{i,j}$ is 
\[
\sum_{\substack{k,x,y \\ k+x = i \\ k+y = j}} s_{(n-x-y-k-1, k)}s_{(1^x)}s_{(1^y)}
\]
\end{corollary}
\begin{proof}
This follows directly from Proposition~\ref{vfrob}, Proposition~\ref{fdrsum}, and equation~\ref{indprod}.
\end{proof}
\begin{corollary}
The bigraded Frobenius image of $\textrm{Res}^{\mathfrak{S}_n}_{\mathfrak{S}_{n-1}}(FDR_n)$ is given by
\[
\mathrm{grFrob} (\mathrm{Res}^{\mathfrak{S}_n}_{\mathfrak{S}_{n-1}}(FDR_n); q,t) = (1-qt)\prod_{i=1}^\infty \frac{(1+x_iqz)(1+x_itz)}{(1-x_iz)(1-x_iqtz)} \bigg|_{z^{n-1}}
\]
where the operator $(\cdots)\mid_{z^{n-1}}$ extracts the coefficient of $z^{n-1}$.
\end{corollary}
By Proposition~\ref{frobimg} we have
\begin{equation}
\textrm{grFrob} (\textrm{Res}^{\mathfrak{S}_n}_{\mathfrak{S}_{n-1}}(FDR_n); q,t) = \sum_{i} \sum_j \sum_{\substack{k,x,y \\ k+x = i \\ k+y = j}} s_{(n-x-y-k-1, k)}s_{(1^x)}s_{(1^y)} q^i t^j.
\end{equation}
Applying Jacobi-Trudi \cite{Sagan} to the $s_{(n-x-y-k-1, k)}$ terms on the right gives
\begin{equation}
\sum_{i} \sum_j \sum_{\substack{k,x,y \\ k+x = i \\ k+y= j}} s_{(n-x-y-k-1, k)}s_{(1^x)}s_{(1^y)} q^i t^j  = \sum_{i} \sum_j \sum_{\substack{k,x,y \\ k+x = i \\ k+y= j}} (h_{n-x-y-k-1}h_{k}- h_{n-x-y-k}h_{k-1})e_{x}e_{y} q^i t^j 
\end{equation}
and reindexing sums gives
\begin{equation}
\sum_{i} \sum_j \sum_{\substack{k,x,y \\ k+x = i \\ k+y= j}} h_{n-x-y-k-1}h_{k}e_{x}e_{y} q^i t^j = \sum_{k} h_kq^kt^kz^k\sum_{x}e_xq^xz^x \sum_{y}e_yq^yz^y \sum_m h_mz^m \bigg|_{z^{n-1}}
\end{equation}
and
\begin{equation}
\sum_{i} \sum_j \sum_{\substack{k,x,y \\ k+x = i \\ k+y= j}} h_{n-x-y-k}h_{k-1}e_{x}e_{y} q^i t^j = \sum_{k} h_kq^{k+1}t^{k+1}z^k\sum_{x}e_xq^xz^x \sum_{y}e_yq^yz^y \sum_m h_mz^m \bigg|_{z^{n-1}}
\end{equation}
from which the result follows.

\section{Maximal bidegrees, cyclic sieving and further directions}

Let $X_n$ denote the subset of $\Phi(n)$ corresponding to bidegrees $(i,j)$ where $i+j= n-1$, in other words,
\begin{equation}
X_n =  \bigcup_{2k+x+y = n-1} \Phi(n,k,x,y).
\end{equation}
This set consists of noncrossing set partitions set partitions of $[n]$ in which $n$ is in a block by itself, all other blocks are size 1 or  2, and singleton blocks other than $n$ are labelled by $\theta$ or $\xi'$. The set $\{G_\pi \mid \pi \in X_n\}$ is invariant (up to sign changes) under the action of the cycle $(1,2, \dots, n-1)$, since $n$ is in a block by itself and rotating all elements except $n$ cannot introduce any new crossings. We therefore have the setup for a cyclic sieving result using Springer's theorem of regular elements (Theorem~\ref{fakedegree}).
\begin{theorem}
\label{sieving}
The triple $(X_n, C_{n-1}, q^{\binom{n}{2}}\mathbf{fd}(FDR_n)_{i+j = n-1})$ exhibits the cyclic sieving phenomenon where $C_{n-1}$ is the cyclic group generated by $(1,2,\dots, n-1)$.
\end{theorem}
\begin{proof}
This follows directly from Theorem~\ref{fakedegree}.
\end{proof}
Thiel \cite{Thiel} studied a version of this cyclic action in which rotation does not introduce a sign change, while in our setup it introduces a sign when $n$ is odd. Thiel proved the following cyclic sieving.
\begin{theorem}[Thiel, 2016]
\label{thielsieving}
The triple $(X_n, C_{n-1}, C_n(q))$ exhibits the cyclic sieving phenomenon, where $C_{n-1}$ is the cyclic group generated by $(1,2,\dots, n-1)$ and $C_n(q)$ is the Mac-Mahon $q$-Catalan number, defined by
\[
C_n(q) := \frac{1}{[n+1]_q} \begin{bmatrix}2n\\q \end{bmatrix}_q.
\]
\end{theorem}
Thiel proved his result via direct computation of $C_n(q)$ and enumeration of fixed points instead of using representation theory, so one might wonder if our basis could give an altenate algebraic proof of his result. The expression for Frobenius image given in Corollary~\ref{frobimg} allows for the computation of the fake degree as
\begin{equation}
\textbf{fd}((FDR_n)_{i+j = n-1}) = \sum_{\substack{k,x,y\\2k+x+y = n-1}} \begin{bmatrix}{n-1}\\{2k,x,y}\end{bmatrix}_q C_k(q) q^{k+\binom{x}{2}+\binom{y}{2}} 
\end{equation}
Combining the two cyclic sieving results it must follow that $q^{\binom{n}{2}}\textbf{fd}((FDR_n)_{i+j = n-1})$ is equivalent to $C_n(q)$ modulo $q^{n-1} - 1$. We have had difficulty in determining this equivalence directly, however, so we propose the following problem:
\begin{problem}
Is there a direct computational proof that $q^{\binom{n}{2}}\textbf{fd}((FDR_n)_{i+j=n-1})$ and $C_n(q)$ are equivalent modulo $q^n-1$?
\end{problem}
Such a proof would complete an alternative representation theoretic proof of Thiel's result.

In joint work with Rhoades \cite{me} we developed a similar combinatorial model for the maximal bidegree components of $FDR_n$, with a basis indexed by all noncrossing set partitions. The action of $S_n$ on that basis could be understood in terms of skein-like relations described by Rhoades \cite{Rhoades}. Patrias, Pechenik, and Striker \cite{PPS} independently discovered an alternate algebraic/geometric model for the irreducible submodule of this action generated by singleton-free noncrossing set partitions sitting in the coordinate ring of a certain algebraic variety. They associated to each partition a polynomial in this coordinate ring defined in terms of matrix minors, and showed that these polynomials satisfied the skein relations described in \cite{Rhoades}. This suggests the following problem:
\begin{problem}
Can our basis for $S^{(n-k-1, k)}$ be realized as a set of polynomials, similarly to the methods of Patrias, Pechenik, and Striker \cite{PPS}?
\end{problem}
One reason for thinking an analogous model might exist is that the relation of block operators described in Lemma~\ref{skein} also appears in the maximal bidegree model and corresponds to a certain identity of two-by-two matrix minors in the work of Patrias, Pechenik and Striker. Their construction therefore extends to give a model for the submodule generated by partitions in $\Phi(n)$ for which the block containing $n$ is at most size two, but we have as yet been unable to discover a treatment of larger blocks satisfying our other relations.
\section{Acknowledgements}
We are very grateful to Brendon Rhoades for many helpful discussions and comments on this project.
  
\end{document}